\documentclass{huber_article}

\RequirePackage[OT1]{fontenc}
\RequirePackage{amsthm,amsmath}
\RequirePackage[numbers]{natbib}
\RequirePackage[colorlinks,citecolor=blue,urlcolor=blue]{hyperref}

\usepackage{tikz}
\numberwithin{equation}{section}
\theoremstyle{plain}
\newtheorem{thm}{Theorem}[section]
\newtheorem{lem}{Lemma}[section]
\newtheorem{cor}{Corollary}[section]

\begin{document}

\title{Random construction of interpolating sets for high dimensional
integration}

\maketitle

{\bfseries \sffamily Mark L. Huber} \par
{\slshape Claremont McKenna College} \par
{\slshape mhuber@cmc.edu}
\vskip 1em
{\bfseries \sffamily Sarah Schott} \par
{\slshape Duke University} \par
{\slshape schott@math.duke.edu}
\vskip 1em

\begin{abstract}
Many high dimensional integrals can be reduced
to the problem of finding the relative measures of two sets.
Often one set will be exponentially larger than the other, making
it difficult to compare the sizes.
A standard method of dealing with this problem is to 
interpolate between the sets with a sequence of nested sets where neighboring
sets have relative measures bounded above by a constant.  Choosing 
such a well balanced sequence
can be very difficult in practice.  Here a new approach that 
automatically creates such sets is presented.  These well balanced
sets allow for faster approximation algorithms for integrals and sums, and 
better tempering and annealing Markov chains for generating random samples.
Applications such as finding the partition function of the Ising model and
normalizing constants for posterior distributions in Bayesian methods
are discussed.
\end{abstract}



\section{Introduction}

\newcommand{\tpa}{{\sffamily TPA}}

Monte Carlo methods for numerical integration can have enormous variance
for the types of high dimensional problems that arise in statistics and
combinatorial optimization applications.  Consider a state space $\Omega$
with measure $\mu$, and $B \subset \Omega$ with finite measure.
Then the problem considered here is approximating
\begin{equation}
\label{EQN:initial}
Z = \int_{x \in B} d\mu(x).
\end{equation}

The classical Monte Carlo approach is to create a random variable $X$
such that $E(X) = Z$ where $X$ has variance as small as possible.
Unfortunately, it is often not possible to know the variance of $X$ 
ahead of time, and this must be estimated as well.  How good the estimate
of the variance is depends on even higher moments which are even more difficult
to estimate.

The method presented here creates an estimate of $Z$ of the form 
$e^{X/k}$, where $k$ is a known constant and $X$ is a Poisson random
variable with mean $k\ln(Z)$.  Because the mean and variance for a 
Poisson random variable are the same, we simultaneously obtain our 
estimate of $Z$ and knowledge of the variance of our estimate.

In fact, the output from our method does the following:
\begin{itemize}
\item{Estimate $Z$ to within a specified relative error with a specified
failure probability in time $O(\ln(Z)^2)$.}
\item{Create a well balanced sequence of nested sets useful in building
annealing and tempering Markov chains that can be used to generate 
Monte Carlo samples.}
\item{Develop an omnithermal approximation for partition functions 
arising from spatial point processes and Gibbs distributions.}
\end{itemize}

\paragraph{Previous work} The new method presented here follows
a long line of work using interpolating sets.  For instance, Valleau and 
Card~\cite{valleauc1972} introduced what they called {\em multistage sampling}
where an intermediate distribution was added to make estimation more
effective.
Jerrum, Valiant and Vazirani~\cite{jerrumvv1986} used a similar idea of
{\em self-reducibility}, and carefully analyzed the computational complexity
of the resulting approximation method.

Suppose we are given two finite sets $B'$ and $B$ such that 
$B' \subset B$ and $\#B$ (the number of elements of $B$), is known.
One way of viewing self-reducibility, is that it effectively requires 
a sequence of sets $B = B_0 \supset B_1 \supset B_2 \supset 
 \cdots \supset B_\ell = B'$ such that 
the relative sizes of the sets $\#B_{i+1}/\#B_{i} \geq \alpha$ 
for a fixed constant $\alpha \in (0,1)$.
Then an unbiased estimate $\hat b_i$ of $\#B_{i+1}/\#B_{i}$ is created for each 
$i$.  The product of these estimates will then be an unbiased estimator for
$\#B' / \#B$, and multiplying by $\#B$ gives the final estimate of $\#B$.

For fixed $\alpha \in (0,1)$,
it is easy to estimate $\#B_{i+1}/\#B_i$ with small relative error
simply by drawing samples from $\#B_{i}$ and counting the percentage
that fall in $\#B_{i+1}$.  The relative standard deviation of a 
Bernoulli random variable with parameter $\alpha$ is $(1 - \alpha)/\alpha$,
so it is important not to make $\alpha$ too small.  On the other hand,
if $\alpha$ is too large, then the nested sets are not shrinking much at
each step, and it will require a lengthy sequence of such sets.  To be
precise,
the number of sets $\ell$ must satisfy 
$\ell \geq \ln_\alpha(\#B'/\#B) = \ln(\#B/\#B')/\ln(\alpha^{-1})$
which goes to infinity as $\alpha$ goes to 1. 
Balancing these two considerations leads
to a optimal $\alpha$ value of about $0.2031$.

The difficulty in applying self-reducibility is 
finding a sequence of sets
such that $\#B_{i+1} / \#B_{i}$ is provably at least $\alpha$, but not so
close to 1 that the sequence of sets is too long.  Ideally, 
$\#B_{i+1} / \#B_i$ would equal $\alpha$ for every $i$, or at least be very
close.  For fixed constants $\alpha_1$ and $\alpha_2$, 
refer to a sequence of sets where the ratios 
$\#B_{i+1} / \#B_i$ fall in $[\alpha_1,\alpha_2]$ for all $i$ 
as {\em well-balanced}.

Well-balanced sequences have other uses as well.  Methods of designing
Markov chains such as simulated annealing~\cite{kirkpatrickgv1983},
simulated tempering, and 
parallel tempering~\cite{swendsenw1986,geyer1991,marinarip1992} all require 
such a sequence of well-balanced sets in order to mix rapidly 
(see~\cite{huber2009a,huber2009b}.)

Now consider the special case of (\ref{EQN:initial}) where $Z$ is the
normalizing constant of a posterior distribution of a Bayesian analysis. 
Skilling~\cite{skilling2006} introduced {\em nested sampling} as a way
of generating a random sequence of nested sets.  The advantage this
method has over self-reducibility is that there is no need to have the
sequence of sets in hand ahead of time.  Instead, it builds up sets from
scratch at random according to a well-defined procedure.  

The disadvantage is that it loses the property of self-reducible algorithms
that the variance of
 the output could be bounded prior to running the
algorithm.  Because deterministic numerical integration was used
in the method, the variance can be determined only up to a factor that 
depends upon the derivatives of a function that is difficult to compute.
Therefore, nested sampling falls in the class of methods where the variance
must be estimated, rather than bounded 
ahead of time as with self-reducibility.

\paragraph{The Tootsie Pop Algorithm}  The method presented here is called
{\em The Tootsie Pop Algorithm} (\tpa), and combines the tight analysis of 
self-reducibility by adding features similar to nested sampling.  
Like self-reducibility, it is 
very general, working over a wide variety of problems.  This includes
the nested sampling domain of Bayesian posterior normalization, but also
includes many other problems where self-reducibility has been applied such
as the Ising model.  Portions of this work were presented at the 
Ninth Valencia International Meetings on
Bayesian Statistics, and also appears in the conference 
proceedings~\cite{huber2010b} with a discussion.

The name is somewhat unusual, and references an advertising campaign run
for Tootsie Pop candies.  A Tootsie Pop is a 
chocolate chewy center surrounded by a candy shell.  The ad campaign
asked ``How many licks does it take to get to the center of a Tootsie Pop?''.
Our algorithm operates in a similar fashion.  Our set $B$ is 
slowly whittled away until the center $B'$ is reached.  The number of 
steps taken to move from $B$ to $B'$ will be Poisson with mean
$\ln(\mu(B)/\mu(B'))$, thereby allowing approximation of $\mu(B)/\mu(B')$.
Therefore, the ``number of licks'' is exactly what is needed to 
form our estimate!

\subsection{Organization}
Section~\ref{SEC:tpa} describes the {\tpa} procedure in detail, then
Section~\ref{SEC:applications} shows some applications.  
Section~\ref{SEC:running} then analyzes the expected running time of the
method, and introduces a two phase approach to {\tpa}.
Section~\ref{SEC:balanced} describes how {\tpa} can be used to build
well-balanced nested sets for tempering.  Section~\ref{SEC:omnithermal}
shows how to create an approximation that simultaneously works for 
all members of a continuous family of sets at once.  Finally, 
Section~\ref{SEC:conclusions} discusses further areas of exploration
with {\tpa} techniques.

\section{The Tootsie Pop Algorithm}
\label{SEC:tpa}

The {\tpa} method has four general ingredients:
\begin{enumerate}
\item{A measure space $(\Omega,{\cal F},\mu)$}
\item{Two finite measurable sets $B$ and $B'$ satisfying $B' \subset B$ and
$\mu(B') > 0$.
The set $B'$ is the {\em center} and $B$ is the {\em shell}.}
\item{A family of nested sets 
$\{A(\beta):\beta \in {\mathbf R} \cup \{\infty\}\}$ such
that $\beta' < \beta$ implies $A(\beta') \subseteq A(\beta)$, 
$\mu(A(\beta))$ is a continuous function of $\beta,$ and
the limit of $\mu(A(\beta))$ as $\beta$ goes to $-\infty$ is 0.}
\item{Special values $\beta_B$ and $\beta_{B'}$ that satisfy
$A(\beta_B) = B$ and $A(\beta_{B'}) = B'$.}
\end{enumerate}

With these ingredients, the {\tpa} method is very simple to describe.
\begin{enumerate}
\item{Start with $i = 0$ and $\beta_i = \beta_{B}$.}
\item{\label{ITM:sample} 
Draw a random sample $Y$ from $\mu$ conditioned to lie in $A(\beta_i)$.}
\item{Let $\beta_{i+1} = \inf\{\beta:Y \in A(\beta)\}$.}
\item{If $Y \in B'$ stop and output $i$.}
\item{Else set $i$ to be $i + 1$ and
go back to step 2.}
\end{enumerate}
Another way of describing the draw in Step 2 is that for measurable $D$, 
$P(Y \in D) = \mu(D \cap A(\beta_i))/\mu(A(\beta_i)).$  At each step, 
the set $A(\beta_i)$ shrinks with probability 1, and so is slowly worn away
until the sample falls into the region $B'$.
 
Line~\ref{ITM:sample} above deserves special attention.  Drawing
a random sample $Y$ from $\mu$ conditioned to lie in $A(\beta_i)$ is in
general a very difficult problem.  The good news is that the importance of 
this problem means that a vast literature
for solving this problem exists.  Markov chain Monte Carlo (MCMC)
methods are critical to obtaining these samples, and variations on
the early methods 
have blossomed over the last fifty years.  Readers are referred
to~\cite{shonkwilerm2009,robertc2010,fishman1994} and the references
therein for more information.

Of course, any other method for turning samples into approximations 
either implicitly or explicitly depend on the ability to execute 
some variant of line~\ref{ITM:sample} as well, 
so our algorithm is not actually demanding anything
above and beyond what others require.  The algorithm is easily modified
to handle different methods of simulating random variables.  For instance,
nested sampling~\cite{skilling2006} draws several such $Y$ variables at
once, and \tpa{} can be written to do so as well.  

The key fact about this process is the following:
\begin{thm}
At any step of the algorithm, let 
\[
E_i = \ln(\mu(A(\beta_i))) -\ln(\mu(A(\beta_{i+1}))).
\]  
Then the $E_i$
are independent, identically distributed exponential random variables
with mean 1.
\end{thm}

\begin{proof}  To simplify the notation, let $m(b) = \mu(A(b))$.
Begin by showing that each $U_i = m(\beta_{i+1})/m(\beta_i)$ is 
uniform over $[0,1]$.  Fix $\beta_i \geq \beta_{B'}$ and let $a \in (0,1)$.  Then
since $m(b)$ is a continuous function in $b$ with 
$\lim_{b \rightarrow -\infty} m(b) = 0$, there exists a 
$b \in (-\infty,\beta_i]$ such that 
$m(b) / m(\beta_i) = a.$  Call this value $\beta_a$.

Let $0 < \epsilon < m(\beta_B) - a.$  Then by the same reasoning 
there is a value 
$\beta_{a + \epsilon} \leq \beta_B$ such that 
$m(\beta_{a + \epsilon}) / \mu(\beta_i) = a + \epsilon.$  Now consider
$Y$ drawn from $\mu$ conditioned to lie in $A(\beta_i)$.  Then 
$\beta_{i+1} = \inf\{b:Y \in A(b)\}$.  

Moreover, $\{U_i \leq a\} \Rightarrow \{Y \in A(\beta_a)\},$ an
event which occurs with probability $m(\beta_a)/m(\beta_i) = a$.  So
$P(U_i \leq a) \geq a$.

On the other hand, $Y \notin A(\beta_{a + \epsilon})$ implies 
$\beta_{i+1} \geq \beta_{a + \epsilon}$ which means that 
$U_i = m(\beta_{i+1}) / m(\beta_i) \geq a + \epsilon$.  In other words,
$P(U_1 < a + \epsilon) \leq P(Y \in A(\beta_{a + \epsilon})) = a + \epsilon$.
This holds for $\epsilon$ arbitrarily small, hence by dominated convergence
$P(U_i \leq a) \leq a.$
Therefore, $P(U_i \leq a)$ is at least and at most $a$, so 
$P(U_i \leq a) = a$, which shows that $U_i$ is uniform on $[0,1]$.

Observing that the negative of the natural log of a uniform number of $[0,1]$
is an exponential with mean 1 completes the proof.
\end{proof}

For $t(b) = \ln(\mu(A(b)))$, the theorem says the points
$t(\beta_0),t(\beta_1),\ldots,t(\beta_i)$ in a run of {\tpa} are separated
by exponential random variables of mean 1, in other words, these points
form a homogeneous Poisson point process on 
$[t(\beta_{B'}),t(\beta_B)]$ of rate 1.

\subsection{Taking advantage of Poisson point processes}
The first application of this is to describe the 
total number of points used by a run of {\tpa},
that is, the value of $i$ at the end of the 
algorithm.  Because the $t(\beta_i)$ values form a Poisson point process, the
distribution of $i$ is Poisson with mean 
$t(\beta_B) - t(\beta_{B'}) = \ln(\mu(B)/\mu(B')).$

Furthermore, the union of $k$ independent Poisson point processes of rate
1 is also a Poisson point process of rate $k$.  That means that after $k$
runs of {\tpa}, the distribution of the total number of samples used is
Poisson with mean $k\ln(\mu(B)/\mu(B')).$

\section{Applications}
\label{SEC:applications}
The following examples illustrate some of the uses for {\tpa}.

\subsection{The Ising model}  The Ising model is an example of a 
{\em Gibbs distribution}, where a function 
$H:\Omega \rightarrow \mathbf{R}$ gives rise
to a distribution on $\Omega$:
\begin{equation}
\pi(x) = \frac{1}{Z(\beta)} \exp(-\beta H(x)).
\end{equation}
From applications in statistical physics, 
$\beta \geq 0$ is known as the inverse temperature,
and $Z(\beta)$ is called the partition function.

In the Ising model, each node of a graph $G = (V,E)$ is assigned one
of two values.  There are many ways to represent the model.  
In the form considered here, each node is either 0 or 1,
and for $x \in \{0,1\}^V$, $-H(x)$ is one plus 
the number of edges $e \in E$ such 
that the endpoints of the edge have the same value in $x$.  This can be
written as
$H(x) = -[1 + \sum_{\{i,j\} \in E} (1 - x(i) - x(j) + 2 x(i) x(j))].$

In order to embed this problem in the framework of {\tpa}, add an
auxiliary dimension to the configuration $x$.  The auxiliary state space
is 
\[
\Omega_{\textrm{aux}}(\beta) = \{(x,y):x \in \{0,1\}^V,y\in[0,\exp(-\beta H(x))\}.
\]
Some notes on $\Omega_{\textrm{aux}}(\beta)$:
\begin{itemize}
\item{The total length of the line segments in $\Omega_{\textrm{aux}}(\beta)$ is
just $Z(\beta)$.  That is to say, $\mu(\Omega_{\textrm{aux}}(\beta)) = Z(\beta)$
where $\mu$ is the one dimensional Lebesgue measure of the union of the line
segments.}
\item{Let $\beta' < \beta$.  Then since $-H(x) > 0$, 
$\Omega_{\textrm{aux}}(\beta') \subset \Omega_{\textrm{aux}}(\beta)$.  Moreover,
$Z(\beta)$ is a continuous function that goes to 0 as 
$\beta \rightarrow -\infty$.  Therefore Condition 
2 of the {\tpa} ingredients is satisfied.}
\item{For $\beta = 0$, $y \in [0,1]$ for all $x \in \{0,1\}$.  That means
$Z(0) = 2^V$.}
\item{Let $\beta > 0$.  Then $\Omega_{\textrm{aux}}(\beta)$ is the 
shell, and $\Omega_{\textrm{aux}}(0)$ is the center.}
\end{itemize}

With this in mind, the {\tpa} algorithm works as follows.
\begin{enumerate}
\item{Start with $i = 0$ and $\beta_i = \beta$.}
\item{Draw a random sample $X$ from $\pi_{\beta_i}$, then
draw $Y$ (given $X$) uniformly from $[0,\exp(-\beta_i H(X))]$.}
\item{Let $\beta_{i+1} = \ln(Y)/(-H(X))$}
\item{If $\beta_{i+1} \leq 0$ stop and output $i$.}
\item{Else set $i$ to be $i + 1$ and
go back to step 2.}
\end{enumerate}

One run of {\tpa} will require on average  
$1 + \ln(Z(\beta)/Z(0)) = 1 + \ln(Z(\beta)) - \#V \ln(2)$ 
samples from various values
of $\beta$, where $\#V$ is the number of vertices of the graph.

This method of adding an auxiliary variable allows {\tpa} to be used
on a variety of discrete distributions by changing the measure to one
that varies continuously in the index.

\subsection{Posterior distributions}  In Bayesian analysis, often it
is necessary to find the normalizing constant of a posterior distribution.
This is known as the {\em evidence} for a model, and can be written:
\[
Z = \int_{x \in \Omega} f(x) \ dx,
\]
where $f(x)$ is a nonnegative density (the product of the prior density
and the likelihood of the data)  and $\Omega \subseteq \mathbf{R}^n$.

For a point $c \in \Omega$ and $\epsilon > 0$, let $B^1_{\epsilon}(c)$ be the 
points within $L_1$ distance $\epsilon$ of $c$.
Suppose that for a particular $c$ and $\epsilon$, 
$B^1_{\epsilon}(c) \subset \Omega$ and
there is a known $M$ such that 
$(1/2) M \leq f(x) \leq M$ for all $x \in B^1_{\epsilon}(c)$.

Then to estimate $Z(\epsilon) = \int_{x \in B^1_{\epsilon}(c)} f(x) \ dx$, 
draw $N$ iid samples
$X_1,\ldots,X_N$ 
uniformly from $B_\epsilon(c)$, and let the estimate be
$\hat Z(\epsilon) = (2\epsilon)^{-n} \sum_{i} f(X_i) / N$.  
Then $\hat Z(\epsilon)$ is an
unbiased estimate for $Z(\epsilon)$ with standard deviation bounded above
by $Z(\epsilon)/\sqrt{k}$.

Now the connection to {\tpa} can be made.  The family of sets will be
$\{A(\beta) = B^1_{\beta}(c) \cap \Omega\}$, and the measure is
$\mu(A(\beta)) = \int_{x \in A(\beta)} f(x) \ dx$.  
The shell will be $A(\infty)$
(so $Z = \mu(A(\infty))$) 
and the center $A(\epsilon)$ (with measure $Z(\epsilon)$.)  
{\tpa} can then be used to estimate 
$Z / Z(\epsilon)$, and the estimate of $Z(\epsilon)$ can then finish 
the job.

\section{Running time of {\tpa}}
\label{SEC:running}

Suppose that {\tpa} is run $k$ times, and the $k$ values of the $i$ variable at
the end of each run are summed together.  Call this sum $N$.  
Then $N$ has a Poisson distribution with
mean $k\ln(\mu(B)/\mu(B')).$  This makes $N / k$ an unbiased estimate
of $\ln(\mu(B)/\mu(B')).$  The variance of $N/k$ is
$\ln(\mu(B)/\mu(B'))/k$.

Let $W$ be a normal random variable of mean 0 and variance 1,
and $W_{\alpha}$ be the inverse cdf of $W$ so that 
$\textrm{Pr}(W \leq W_{\alpha}) = \alpha$.
Then the normal approximation to the Poisson gives 
\begin{equation}
\left[(N/k) - W_{\alpha/2} \sqrt{N/k},(N/k) + W_{\alpha/2} \sqrt{N/k} \right]
\end{equation}
as an
approximately $1 - \alpha$ level confidence interval for $\ln(\mu(B)/\mu(B')).$
Exponentiating then gives the $1 - \alpha$ level for $\mu(B)/\mu(B')$.

For a specific output, it is also possible to build an exact confidence
interval for $\mu(B)/\mu(B')$ 
since the distribution of the output is known exactly.

Similarly, it is easy to perform a Bayesian analysis and find a
credible interval given a prior on $\ln(\mu(B)/\mu(B')).$  

Lastly, consider how to build an 
$(\epsilon,\delta)$ {\em randomized approximation scheme} 
(RAS) whose
output $\hat A$ satisfies:
\[
\textrm{Pr}\left((1+\epsilon)^{-1} \leq \frac{\hat A}
  {\mu(B)/\mu(B')} < 1 + \epsilon\right)
 > 1 - \delta.
\]

For simplicity, assume that 
$\mu(B) / \mu(B') \geq e$.  Note that when $\mu(B) / \mu(B') < e$, then
simple acceptance rejection can be used to obtain an $(\epsilon,\delta)$-RAS 
in $\Theta(\epsilon^{-2}\ln(\delta^{-1}))$ time.  See 
\cite{fishman1996} for a description of this method.

The following lemma gives a bound on the 
tails of the Poisson distribution.
\begin{lem}
\label{LEM:chernoff}
Let $\tilde\epsilon > 0$ and $N$ be a 
Poisson random variable with mean $k \lambda$, 
where $\tilde\epsilon / \lambda \leq 2.3$.  
Then
\[
\textrm{Pr}\left( \left|\frac{N}{k} - \lambda\right| \geq \tilde\epsilon \right)
 \leq 2 \exp\left(-\frac{k\tilde\epsilon^2}{2\lambda}
  \left(1 - \frac{\tilde\epsilon}{\lambda}\right)
 \right).
\]
\end{lem}
(This result is a special case of Theorem~\ref{THM:process} shown later.)

To obtain our $(\epsilon,\delta)$-RAS, it is sufficient to make
$\tilde\epsilon = \ln(1 + \epsilon)$, and to choose $k$ so that 
$2 \exp(-k\tilde\epsilon^2(1 - \tilde\epsilon/\lambda)/[2\lambda]) 
  \leq \delta$, where
$\lambda = \ln(\mu(B)/\mu(B')).$  This is
made more difficult by the fact that $\lambda$ is unknown at the start of the
algorithm!

There are many ways around this difficulty, perhaps the simplest is to
use a two phase method.  
First get a rough estimate of $\lambda$, then refine this estimate to 
the level demanded by $\epsilon.$

\begin{description}
\item[Phase I] Let $\epsilon_a = \ln(1 + \epsilon)$ and  
$k_1 = 2 \epsilon_a^{-2} (1 - \epsilon_a)^{-1}\ln(2\delta^{-1}).$ 
Then let $N_1$ be the sum
of the outputs from $k_1$ runs of {\tpa}.
\item[Phase II] Set $k_2 = N_1(1 - \epsilon_a)^{-1}.$  Let
$N_2$ be the sum of the outputs from $k_2$ runs of {\tpa}.  The final
estimate is $\exp(N_2 / k_2)$.
\end{description}

Phase I estimates $\lambda$ to within an additive error
$\epsilon_a \lambda$.
Phase II uses the Phase I estimate of $\lambda$ to create a better 
estimate of $\lambda$
to within an additive error of $\epsilon_a$.  Note that 
$\epsilon_a \approx \epsilon$ in the sense that 
$\lim_{\epsilon \rightarrow 0} \epsilon_a / \epsilon = 1$.

\begin{thm}
The output $\hat A$ of the above procedure is an $(\epsilon,\delta)$ 
randomized approximation scheme for $\mu(B)/\mu(B').$ 
The running time is random, with an expected running time that is 
$\Theta((\ln(\mu(B)/\mu(B')))^2 \epsilon^{-2}\ln(\delta^{-1})).$
\end{thm}

\begin{proof} 
Call Phase I a success if $N_1/k_1$ is within distance $\epsilon_a\lambda$
of $\lambda$.
From Lemma~\ref{LEM:chernoff} with $\tilde\epsilon = \epsilon_a \lambda$:
\[
\textrm{Pr}\left( \left|\frac{N_1}{k_1} - \lambda\right| 
  \geq \epsilon_a \lambda \right)
 = 2 \exp\left( - k_1 (\lambda \epsilon_a^2)(1 - \epsilon_a)\right)
 \leq \delta/2
\]
since $\lambda \geq 1$ and 
$k_1 = \epsilon_a^{-2}(1-\epsilon_a)^{-1} \ln(2\delta^{-1}).$
Therefore, the probability that Phase I is a failure is at most 
$\delta/2$.

When Phase I is a success, 
$(1 -\epsilon_a ) \lambda k_1  \leq N_1$.  
In this event $k_2 = N_1 (1 - \epsilon_a)^{-1} 
\geq \lambda k_1 = \lambda \epsilon_a^{-2} (1 - \epsilon_a)^{-1} 
  \ln(2\delta^{-1}).$ 
Plugging this in to
Lemma~\ref{LEM:chernoff} yields:
\[
\textrm{Pr}\left( \left|\frac{N_2}{k_2} - \lambda\right| 
  \geq \epsilon_a \lambda \right)
 \leq 2 \exp\left( - \frac{\lambda \epsilon_a^{-2} (1 - \epsilon_a)^{-1}
 \ln(2\delta^{-1})
 \epsilon_a^2}{2 \lambda}\left(1 - \frac{\epsilon_a}{\lambda}\right) 
 \right).
\]
Using $\lambda \geq 1$, the right hand side is at most $2\delta$.

The chance of failure in either Phase is at most 
$\delta/2 + \delta/2 = \delta,$ so altogether 
$|(N_2/k_2) - \lambda| \leq \epsilon_a$ with probability at least 
$1 - \delta$.  Exponentiating then gives
\[
(1 + \epsilon)^{-1} = e^{-\epsilon_a} \leq \exp(N_2/k_2)/\lambda 
  \leq e^{\epsilon_a} = 1 + \epsilon
\]
with probability at least $1 - \delta$.

The expected number of samples needed in Phase I is $k_1 \lambda$, while
the expected number needed in Phase II is:
\begin{align*}
\textrm{E}(N_2) 
&= 
 \textrm{E}(\textrm{E}(N_2|N_1))
 = \textrm{E}((1 - \epsilon_a)^{-1} N_1 \lambda)
 = 2(1 - \epsilon_a)^{-2} \epsilon_a^{-2}\ln(2\delta^{-1})\lambda^2.
\end{align*}
Since $\epsilon_a = \Theta(\epsilon)$, the proof is complete.
\end{proof}

\section{Well-balanced nested sets}
\label{SEC:balanced}

Consider running {\tpa} $k$ times, and collecting
all the values of $\beta_i$ generated during these runs.  Let $P$ denote
this set of values, then $P$ forms 
a Poisson point process of rate $k$ on $[\beta_{B'},\beta_B].$

Call $\beta_B = \alpha_0 > \alpha_1 > \cdots > \alpha_{\ell} = \beta_{B'}$
a {\em well-balanced cooling schedule} if 
$\mu(A(\alpha_{i+1})) / \mu(A(\alpha_i))$ is close to $1/e$
for all $i$ from 0 to $\ell - 1$.

Given $P$, finding such a well-balanced set is easy: simply order the 
$\beta$ values in $P$, and set $\alpha_i = \beta_{(ik)}$.  The value of 
$\ln(\mu(A(\alpha_{i+1}) / A(\alpha_{i})))$ will have distribution equal to 
the sum of $k$ iid exponential random variables with mean $1 / k$.   
So $\ln(\mu(A(\alpha_{i+1}) / \mu(A(\alpha_{i}))))$ will be gamma distributed
with mean $1$ and standard deviation $1 / k$.

\section{Omnithermal approximation}
\label{SEC:omnithermal}

Suppose instead of just a single value of interest 
$\mu(B)/\mu(B')$, it is necessary 
to create an approximation of $\mu(A(\beta))/\mu(B')$ that is valid for all
values $\beta \in [\beta_{B'},\beta_B]$ simultaneously.  Call this an 
{\em omnithermal 
approximation}.  These problems appear in what are called doubly intractable
posterior distributions arising in Bayesian analyses involving spatial
point processes.  They 
are usually dealt with indirectly using Markov chain Monte
Carlo with auxiliary variables~\cite{murraygm2006}, 
but omnithermal approximation allows for a 
more direct approach.

In the last section the Poisson point process $P$ formed from the $\beta_i$
values collected from $k$ runs of {\tpa} was introduced.
To move from $P$ to a Poisson process, set
\[
N_P(t) = \#\{b \in P:b \geq \beta_B - t\}.
\]
As $t$ advances from 0 to $\beta_B - \beta_{B'}$, $N_P(t)$ increases by
1 whenever it hits a $\beta$ value.  By the theory of Poisson point processes,
this happens at intervals that will be independent exponential random
variables with rate $k$.

Given $N_P(t)$, approximate $\mu(B)/\mu(A(\beta))$ by 
$\exp(N_P(\beta_B - \beta)/k)$.  When $\beta = \beta_{B'}$, this is just
the approximation given earlier, so this generalizes the description
of {\tpa} from before.

The key fact is that $N_P(t) - kt$ is a right continuous martingale.
To bound the error in $\exp(N_P(t)/k)$, it is necessary to
bound the probability that $N_P(t) - kt$ has drifted too far away from 0.

\begin{thm}
\label{THM:process}
Let $\tilde\epsilon > 0$.  Then for $N_P(\cdot)$ a rate $k$ Poisson process
on $[0,\lambda]$, where $\tilde\epsilon / \lambda \leq 2.3$:
\[
\textrm{Pr}\left(\sup_{t \in [0,\lambda]} \left|\frac{N_P(t)}{k} - t \right|
 \geq \tilde\epsilon \right) \leq 2 \exp\left( 
     - \frac{k\tilde\epsilon^2}{2\lambda}
  \left(1 - \frac{\tilde\epsilon}{\lambda}\right) \right).
\]
\end{thm}

\begin{proof}
The approach will be similar to finding a Chernoff bound~\cite{chernoff1952}.
Since $\exp(\alpha x)$ is convex for any positive constant $\alpha$, and
$N_P(t)$ is a right continuous martingale, $\exp(\alpha N_P(t))$
is a right continuous submartingale.

Let $A_U$ denote the event that $(N_P(t)/k) - t > \epsilon$
for some $t \in [0,\lambda]$.  Then for all $\alpha > 0$:
\[
\textrm{Pr}(A_U) = \textrm{Pr}\left(\sup_{t \in [0,\lambda]}
  \exp(\alpha N_P(t)) \geq 
   \exp(\alpha k t + \alpha k \epsilon) \right).
\]
It follows from basic Markov-type inequalities on right
continuous submartingales (p. 13 of \cite{karatzass1991})
that this probability can be upper bounded as
\[
\textrm{Pr}(A_U) \leq \textrm{E}(\alpha \exp(N_P(\lambda)) / 
  \exp(\alpha k \lambda
  + \alpha k \tilde\epsilon)).
\]
Using the moment generating function for a Poisson
with parameter $k\lambda$:
\[
\textrm{E}[\exp(\alpha N_P(\lambda))] = \exp(k \lambda(\exp(\alpha) - 1)),
\]
which means
\[
\textrm{Pr}(A_U) \leq \exp(\lambda(e^\alpha - 1 - \alpha) 
  + \alpha \tilde\epsilon)^k.
\]
A Taylor series expansion shows that 
$e^\alpha - 1 - \alpha \leq (\alpha^2/2)(1 + \alpha)$
as long as $\alpha \in [0,2.31858...]$.  Set
$\alpha = \tilde\epsilon / \lambda$.  Simplifying the resulting
upper bound yields
\[
\textrm{Pr}(A_U) \leq \exp\left(-\frac{k\tilde\epsilon^2}{2\lambda}
 \left(1 - \frac{\tilde\epsilon}{\lambda}\right) \right).
\]
The other tail can be dealt with in a similar fashion,
yielding a bound
\[
\textrm{Pr}\left(\sup_{t \in [0,\lambda]} [N_P(\alpha)/k] - t < \tilde\epsilon\right)
  \leq \exp\left(-\frac{k\tilde\epsilon^2}{2\lambda} \right).
\]
The union bound on the two tails then yields the theorem.
\end{proof}

\begin{cor}
For $\epsilon \in (0,0.3)$, $\delta \in (0,1)$, and 
$\ln (\mu(B)/\mu(B')) > 1$, after  
\[
k = 2 (\ln (\mu(B)/\mu(B')) 
    (3\epsilon^{-1} + \epsilon^{-2}) \ln(2/\delta)
\]
runs of TPA, 
the points obtained can be used to build an 
$(\epsilon,\delta)$ omnithermal approximation.
\end{cor}

\begin{proof}
In order for the final approximation to be within a multiplicative 
factor 
of $1 + \epsilon$ of the true result, 
the log of the approximation must be
accurate to an additive term of $\ln(1 + \epsilon)$.  
Let $\lambda = \ln(\mu(B)/\mu(B'))$, so 
$k = 2\lambda(3\epsilon^{-1} + \epsilon^{-2})\ln(2/\delta)$. 
To prove the corollary from the theorem, it suffices 
to show that 
$2\exp(-2\lambda(3\epsilon^{-1} + 
  \epsilon^{-2}\ln(2/\delta)[\ln(1+\epsilon)]^2
  (1 - \epsilon/\lambda)/(2\lambda))
 < \delta$.
After canceling the factors of $\lambda$, and noting that when
$\lambda > 1$, $1 - \epsilon/\lambda < 1 - \epsilon$, it suffices
to show that 
$(3\epsilon^{-1} + \epsilon^{-2})(1 - \epsilon)
  [\ln(1 + \epsilon)]^2 > 1.$
This can be shown for $\epsilon \in (0,0.3)$ 
by a Taylor series expansion.
\end{proof}

\subsection{Example:  Omnithermal approximation
for the Ising model}

Consider the following model.  The value of $\beta$ is drawn from a 
prior density $f_{\textrm{prior}}(\cdot)$ on $[0,\infty)$, 
and then the data (conditioned on 
$\beta$) is drawn from the Ising model.  This was used by 
Besag~\cite{besag1974} as a model for agriculture wherein soil quality
of adjacent plots was more likely to be similar.

Given the data $X$, 
the posterior in the Bayesian analysis is the following density on $\beta:$
\begin{equation}
\label{EQN:posterior}
f_{\textrm{post}}(b) \propto f_{\textrm{prior}}(b) \frac{\exp(b H(X))}{Z(b)}.
\end{equation}
The {\em evidence} for the model is the integral of the right hand side
of (\ref{EQN:posterior}) as $b$ runs from 0 to $\infty$.  This is only
a one-dimensional integration, and so should be straightforward from
a numerical perspective, except that $Z(b)$ is unknown.

Here is where the omnithermal approximation comes in:  it gives an
approximation for $Z(b)$ that is valid for {\em all} values of $b$ at
once.  Any numerical integration technique can be used, and the final 
value for the evidence (not including error arising from 
the numerical method) will
be within a factor of $1 + \epsilon$ of the true answer.

Figure~\ref{FIG:ising} presents two omnithermal
approximations for $\log Z_\beta$ generated using this 
method on a small $4 \times 4$ square lattice.  
The top graph is the result of a single run of TPA from $\beta = 2$ down to 
$\beta = 0$.  At each $\beta$ value returned by TPA,
the approximation drops by 1.  The bottom graph is
the result of $\lceil \ln(4\cdot 10^6) \rceil = 16$
runs of TPA.  This run told us that 
$Z_2 \leq 217$ with confidence $1 - 10^{-6}/2$.
Therefore, using $\epsilon = 0.1$, and $\delta = 10^6 / 2$ in
Theorem~\ref{THM:process} shows that
$r = 330000$ samples suffice 
for a $(0.1,10^{-6})$ omnithermal approximation.


\begin{center}
\begin{figure}
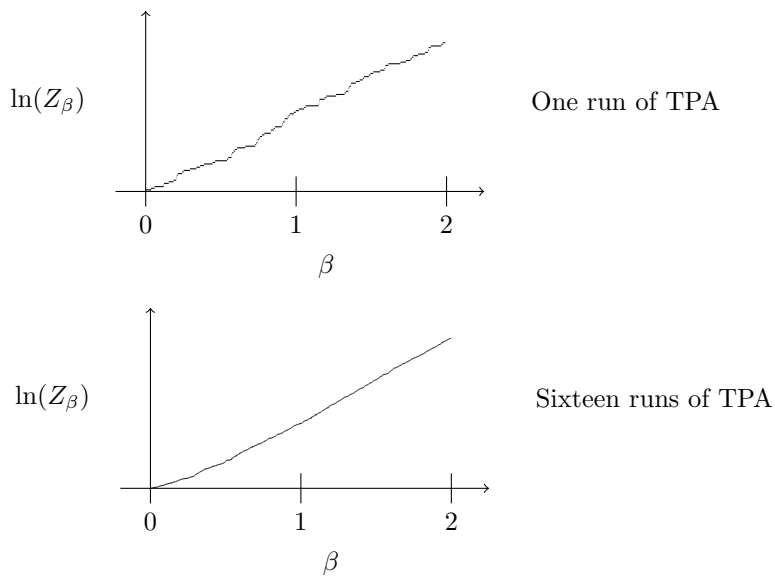

\begin{center}
 
\end{center}
\caption{Omnithermal approximations for the partition function of the
Ising model on
a $4 \times 4$ lattice \label{FIG:ising}}
\end{figure}
\end{center}

\section{Conclusions and further work}
\label{SEC:conclusions}
The strength of {\tpa} is the generality of the procedure, but that same
generality means that it is possible to do better in restricted
circumstances.  For instance, when $f(x)$ falls into the class of 
Gibbs distributions, 
\u{S}tefankovi\u{c} et al.~\cite{stefankovicvv2009} were able to give
an $\tilde O(\ln(Z))$ algorithm for approximating $Z$, but the high constants
involved in their algorithm make it solely of theoretical interest.  
(Here the 
$\tilde O$ notation hides logarithmic factors.)  {\tpa} can be used in
conjunction with their algorithm~\cite{huberPreprint3} 
to build an $O(\ln(Z)\ln(\ln(Z)))$ 
algorithm, and work continues to bring this running time down to 
$O(\ln(Z)).$

\section*{Acknowledgments}

Both authors were supported in this work by NSF CAREER grant
DMS-05-48153.  Any opinions, findings, and conclusions or 
recommendations expressed in this material 
are those of the authors and do not necessarily reflect the 
views of the National Science Foundation.



\bibliographystyle{plain}
\bibliography{../../../2011_refs}

\end{document}